\documentclass[11pt,a4paper,oneside]{amsart}

\newcounter{commentcounter}

\usepackage{amsmath} % Lots of maths functionality
\usepackage{amssymb} % Maths symbols
\usepackage{amsthm} % Maths environments: \begin{proof}, etc.
\usepackage{stmaryrd} % [[ brackets
\usepackage[english]{babel} % Language and hyphenation.
\usepackage[font=small]{caption} % More flexibility for captioning figures
\usepackage[nodayofweek]{datetime}
\usepackage[shortlabels]{enumitem} % Change enumeration labelling with \begin{enumerate}[a)] etc.
\usepackage[T1]{fontenc} % For font encoding, to allow accents, copy & paste, inequality signs, etc. to all work nicely.
\usepackage[utf8]{inputenc} % To be loaded after fontenc, also for encoding.
\usepackage{ifthen} % For Dani's \begin{com} environment and \numberedtheorem
\usepackage{mathabx} % Contains \pnest symbol and more. Clashes with accents for \ring
\usepackage{mathtools} % Uses amsmath, fixes quirks and adds functionality.
\usepackage[dvipsnames]{xcolor} % Allow links to have colour. Needs to be before hyperref and tikz-cd
\usepackage[pdftex,  colorlinks=true, pagebackref=true]{hyperref} % Makes references and citations into links
    \hypersetup{urlcolor=RoyalBlue, linkcolor=RoyalBlue,  citecolor=black}
\usepackage{setspace} % Allows \onehalfspacing etc. for changing gaps between lines
\usepackage{tikz-cd} % Commutative diagrams
\usepackage{xfrac} % Nicer quotients 
\usepackage[capitalize]{cleveref} % use \Cref{} for instead of X~\ref{} 
\usepackage{comment}
\makeatletter
\@namedef{subjclassname@1991}{Mathematical subject classification 1991}
\@namedef{subjclassname@2000}{Mathematical subject classification 2000}
\@namedef{subjclassname@2010}{Mathematical subject classification 2010}
\@namedef{subjclassname@2020}{Mathematical subject classification 2020}
\makeatother

\renewcommand*{\backref}[1]{}
\renewcommand*{\backrefalt}[4]
{
    \ifcase #1
        No citation in the text.
    \or
        Cited on Page #2.
    \else
        Cited on Pages #2.
    \fi
}

\newtheorem{thm}{Theorem}[section]
\newtheorem{lemma}[thm]{Lemma}
\newtheorem{corollary}[thm]{Corollary}
\newtheorem{prop}[thm]{Proposition}

\newtheorem{question}[thm]{Question}

\theoremstyle{definition}
\newtheorem{defn}[thm]{Definition}

\theoremstyle{plain}

    \newtheoremstyle{TheoremNum}
        {\topsep}{\topsep} %%% space between body and thm
        {\itshape} %%% Thm body font
        {-0.25cm} %%% Indent amount (empty = no indent)
        {\bfseries} %%% Thm head font
        {.} %%% Punctuation after thm head
        { }  %%% Space after thm head
        {\thmname{#1}\thmnote{ \bfseries #3}}%%% Thm head spec
    \theoremstyle{TheoremNum}

\newcommand*{\claimproofname}{My proof}

\DeclareMathOperator{\Aut}{\mathrm{Aut}}

\DeclareMathOperator{\Isom}{\mathrm{Isom}}

\DeclareMathOperator{\cay}{Cay}

\newcommand{\cald}{{\mathcal{D}}}

\newcommand{\calg}{{\mathcal{G}}}

\newcommand{\call}{{\mathcal{L}}}
\newcommand{\calm}{{\mathcal{M}}}

\newcommand{\calq}{{\mathcal{Q}}}

\newcommand{\cals}{{\mathcal{S}}}
\newcommand{\calt}{{\mathcal{T}}}
\newcommand{\calu}{{\mathcal{U}}}
\newcommand{\calv}{{\mathcal{V}}}
\newcommand{\calw}{{\mathcal{W}}}

\newcommand{\PSL}{\mathrm{PSL}}

\newcommand{\SU}{\mathrm{SU}}

\newcommand{\Sp}{\mathrm{Sp}}

\newcommand{\SO}{\mathrm{SO}}

\newcommand{\CAT}{\mathrm{CAT}}

\newcommand{\onto}{\twoheadrightarrow}

\DeclareMathOperator{\CT}{CT}

\newcommand{\EE}{\mathbb{E}}
\newcommand{\RH}{\mathbf{H}_\mathbb{R}} % Real hyperbolic space
\newcommand{\CH}{\mathbf{H}_\mathbb{C}} % Complex hyperbolic space
\newcommand{\HH}{\mathbf{H}_\mathbb{H}} % Quaternion hyperbolic space
\newcommand{\OH}{\mathbf{H}^2_\mathbb{O}} % Cayley hyperbolic space
\newcommand{\Ffour}{\mathrm{F}_4^{-20}} % The other rank one group

\newcommand{\ZZ}{\mathbb{Z}}

\newcommand{\RR}{\mathbb{R}}

\usepackage{tikz}
\usetikzlibrary{arrows,quotes}
\tikzstyle{blackNode}=[fill=black, draw=black, shape=circle]

\usepackage[foot]{amsaddr}
\author{Sam Hughes}
\address[S.~Hughes]{Rheinische Friedrich-Wilhelms-Universit\"at Bonn, Mathematical Institute, Endenicher Allee 60, 53115 Bonn, Germany}
\email{sam.hughes.maths@gmail.com}\email{hughes@math.uni-bonn.de}
\author{Motiejus Valiunas}
\address[M.~Valiunas]{Instytut Matematyczny, Universytet Wroc{\l}awski, plac Grunwaldzki 2/4, 50-384 Wroc{\l}aw, Poland}
\email{motiejus.valiunas@math.uni.wroc.pl}
\title[Asynchronous automaticity and non-positive curvature]{A note on asynchronously automatic groups and notions of non-positive curvature}

\date{\today}
\begin{document}
\begin{abstract}
    We prove groups acting cocompactly on locally finite trees with hyperbolic vertex stabilisers are asynchronously automatic.  Combining this with previous work of the authors, we obtain an example of a group satisfying several non-positive curvature properties (being a $\mathrm{CAT}(0)$ group, an injective group, a hierarchically hyperbolic group, and having quadratic Dehn function) which is asynchronously automatic but not biautomatic.
\end{abstract}
\maketitle

\section{Introduction}
Studying languages and automata related to presentations of groups has been one of the driving motivations of combinatorial and geometric group theory.  This has given rise to many classes of groups --- biautomatic groups, automatic groups, asynchronously automatic groups, semihyperbolic groups (introduced in \cite{AlonsoBridson1995}), and so on.  The reader is referred to \cite{EpsteinEtAl1992} for background on automaticity and \cite{Rees2022} for a more recent survey.  

Applying the theory of languages and automata to groups has seen a number of successes such as: giving effective solutions to the word problem in many $3$-manifold groups \cite{EpsteinEtAl1992} and many other groups (for example mapping class groups \cite{Mosher1995}, $\CAT(0)$ cubical groups \cite{NibloReeves1998}, systolic groups \cite{JanuskiewiczSwicatkowski2006}, Helly groups \cite{ChapolinChepoiGenevoisHiraiOsajda2020}, and Coxeter groups \cite{MunroOsajdaPrzytycki2022,OsajdaPrzytycki2022}); characterising virtually free groups \cite{MullerSchupp1983} and hyperbolic groups \cite{Papasoglu1995,HughesNairneSpriano2022} via languages; as well as elucidating many structural properties of groups admitting stronger language or automata related properties \cite{GerstenShort1991,AlonsoBridson1995}.

There are still large gaps in our understanding of how various forms of non-positive curvature relate with various versions of automaticity.  For example, a recent breakthrough of Leary and Minasyan gave the first examples of $\CAT(0)$ groups which are not biautomatic \cite{LearyMinasyan2021} and an analogous result involving other forms of non-positive curvature was obtained by the authors in \cite{HughesValiunas2022}.  It is still an open question if $\CAT(0)$ groups are necessarily (asynchronously) automatic.

In this note we will examine the class \emph{asynchronously automatic groups}, which we will define in \Cref{sec:prelims}, and its interaction with various classes of non-positively curved groups.  Namely, groups with quadratic Dehn function, $\CAT(0)$ groups, hierarchically hyperbolic groups (HHGs), and groups acting geometrically on injective metric spaces (injective groups).  See \cite{BridsonHaefligerBook,BehrstockHagenSisto2017,BehrstockHagenSisto2019,Lang2013} for definitions of the various classes and \cite{HaettelHodaPetyt2021,HughesValiunas2022} for their interactions.  Our main technical result is a combination theorem relating groups acting on locally finite trees, hyperbolic groups, and the class of asynchronously automatic groups. 

\begin{prop} \label{prop:main}
Let $\Gamma$ be a group.  If $\Gamma$ acts cocompactly on a locally finite tree with hyperbolic vertex stabilisers, then $\Gamma$ is asynchronously automatic.
\end{prop}

The proposition can be applied to prove generalised Baumslag--Solitar groups (GBS$_1$ groups) are asynchronously automatic and to give another proof that hyperbolic-by-free groups are asynchronously automatic (note that a more general result on split extensions of hyperbolic groups was obtained by Bridson \cite{Bridson1993}).  The result can be deduced by combining work of Shapiro \cite{Shapiro1992} and Gersten--Short \cite{GerstenShort1991}.  However, we give a direct proof to make explicit the asynchronous structure.

The asynchronous structures we construct are similar to those constructed recently by Hermiller, Holt, Rees, and Susse \cite{HermillerHoltReesSusse2021} for groups acting on trees, under the assumption that the edge stabilisers admit generating sets satisfying two technical conditions --- \emph{stability} and \emph{limited crossover}. However, the latter condition seems difficult to satisfy in our setting (when the tree is locally finite), and so Proposition~\ref{prop:main} does not follow immediately from the results in \cite{HermillerHoltReesSusse2021}.

We highlight the next corollary (which we prove in \Cref{sec cor proof}) due to its relation to other results in the literature which we will explain below.  Let $H$ be the isometry group of a proper $\CAT(0)$ space $X$.  Recall that a \emph{uniform lattice} $\Gamma$ in $H$ is a discrete subgroup of $H$ such that $X/\Gamma$ is compact.

\begin{corollary} \label{cor lattices}
Let $H$ be one of $\SO(n,1)$, $\SU(n,1)$, $\Sp(n,1)$ or $\Ffour$ with $n\geq 2$ and let $T$ be the automorphism group of a locally finite tree.  Suppose $T$ is non-discrete and cocompact.  If $\Gamma$ is a uniform lattice in $H\times T$, then $\Gamma$ is asynchronously automatic.
\end{corollary}

Any $H$ as in the previous corollary has an associated rank one symmetric space isometric to $\RH^n$, $\CH^n$, $\HH^n$, or $\OH$.  In particular, $\Gamma$ is quasi-isometric to the product of a hyperbolic symmetric space and a tree.  Now, A.~Margolis \cite[Theorem~J]{Margolis2022disc} has proven that any group $\Lambda$ quasi-isometric to a product of hyperbolic graphs $\prod_{i=1}^n X_i$ is biautomatic, provided that none of the $X_i$ are quasi-isometric to a (possibly Euclidean) non-compact symmetric space.  Note that by \cite[Theorem~1.1]{LearyMinasyan2021} and \cite[Theorem~A]{HughesValiunas2022} this result of Margolis is sharp, namely, there exist uniform lattices in both $\Isom(\EE^2)\times T_{10}$ and $\PSL_2(\RR)\times T_{24}$ which are not biautomatic.

Combining the previous corollary with the main result of \cite{HughesValiunas2022} we obtain a group with a strange combination of properties.  In particular, the group $\Gamma$ has very strong non-positive curvature properties: being a hierarchically hyperbolic group, a $\CAT(0)$ group, acting geometrically on an injective metric space, and therefore having quadratic Dehn function.  But $\Gamma$ fails to be biautomatic.  Here we show $\Gamma$ satisfies the weaker property of being asynchronously automatic. This gives the first example of a group satisfying any of the previously mentioned geometric properties which is not biautomatic but is asynchronously automatic.

\begin{thm}\label{cor our boi}
There exists a torsion-free non-residually finite uniform lattice $\Gamma < \PSL_2(\RR)\times T_{24}$ which is a hierarchically hyperbolic group, an injective group, a $\CAT(0)$ group, and is not biautomatic. However, $\Gamma$ is asynchronously automatic.
\end{thm}

As far as the authors are aware, this is one of the first examples of a quadratic Dehn function group which is asynchronously automatic but not biautomatic.  The only other examples are some free-by-cyclic groups claimed to be not automatic by Brady, Bridson, and Reeves \cite{BradyBridsonReeves2006} announced in Bridson's ICM notes \cite{Bridson2006}.

We raise the following set of questions.  

\begin{question} \label{q:async}
Is every (a) hierarchically hyperbolic group, (b) injective group, or (c) $\CAT(0)$ group %, or (d) quadratic Dehn function group
asynchronously automatic?
\end{question}

Note that all of the groups mentioned in Question~\ref{q:async} have a quadratic Dehn function.  However, there exist groups with a quadratic Dehn function, such as the higher Heisenberg groups \cite{Allcock}, which are nilpotent and not virtually abelian, and therefore cannot be asynchronously automatic \cite[Theorem~8.2.8]{EpsteinEtAl1992}.

Towards (a), one may wish to try to extend \Cref{prop:main} from trees to hyperbolic graphs.  The main obstruction to doing this is that one would need a set of preferred paths in the graph along with a compatible regular transversal of the vertex and edge stabilisers.  It is not at all clear to us if this follows from hyperbolicity.  

We do not know of any possible candidate counterexamples to the first two questions; however, the Leary--Minasyan groups introduced in \cite{LearyMinasyan2021} (and related groups \cite{Hughes2021,Valiunas2023,ShepherdValiunas2025}) may be a good starting point to disprove (c).  Note that by \cite{Button2022} the Leary--Minasyan groups are not hierarchically hyperbolic groups.  They were classified up to isomorphism in \cite{Valiunas2022}.

\begin{question}
Which Leary--Minasyan groups are asynchronously automatic?
\end{question}

\subsection*{Acknowledgements}
This work has received funding from the European Research Council (ERC) under the European Union's Horizon 2020 research and innovation programme (Grant agreement No.\ 850930), as well as the National Science Centre (Poland) grant No.\ 2022/47/D/ST1/00779.  The authors would like to thank Martin Bridson and Sarah Rees for helpful correspondence, as well as the anonymous referee for their comments.

\section{Preliminaries} \label{sec:prelims}

\subsection{Automata}
The definitions in this section are standard and have been taken from \cite[Chapters 1 and 7]{EpsteinEtAl1992}.

Let $A$ be a finite set and let $A^\star$ be the free monoid generated by $A$.  We denote the nullstring by $\epsilon$. A \emph{language} over the alphabet $A$ is a subset $\call \subseteq A^\star$. Let $W_L,W_R$ be words over $A$.  A \emph{shuffle} of $(W_L,W_R)$ is a string $W\in A^\star$ and a map $\{1,\dots,|W|\}\to \{L,R\}$ such that if we substitute the nullstring $\epsilon$ in $W$ for each element that maps to $R$ we get $W_L$ and if we substitute $\epsilon$ in $W$ for each element that maps to $L$ we get $W_R$.

\begin{defn}[Finite state automaton]
A \emph{finite state automaton} (FSA) $\calm$ over the alphabet $A$ consists a finite directed graph $\calg(\calm)$, together with a (directed) edge label function $\ell\colon E^+(\calg(\calm)) \to A$, a chosen vertex $o \in V(\calg(\calm))$ called the \emph{initial state} and a subset $F \subseteq V(\calg(\calm))$ of \emph{final states}. The vertices of $\calg(\calm)$ are often referred to as \emph{states}. 

Let $\calm$ be an FSA over an alphabet $A$.  We say a string $W\in A^\star$ is \emph{accepted} by $\calm$ if and only if there is an oriented path $\gamma$ in $\calg(\calm)$ starting from $o$ and ending in a vertex $q \in F$ such that $\gamma$ is labelled by $W$.  A language $\call$ over $A$ is \emph{regular} if and only if there exists an FSA $\calm$ such that $\call$ coincides with the strings of $A^\star$ accepted by $\calm$.  We denote the regular language accepted by $\calm$ by $\call(\calm)$.
\end{defn}

\begin{defn}[Asynchronous automaton]
An \emph{asynchronous (deterministic two-tape) automaton} $\calm$ over $A$ is a partial deterministic automaton over $A\cup\{\$\}$ where the states are partitioned into five subsets, denoted $S_L$, $S_L^\$$, $S_R$, $S_R^\$$ and $S^\$$.  The set $S^\$$ consists of exactly one state $s^\$$ which will be the unique final state for $\calm$.  A directed edge $e$ labelled by an element of $A$ with initial vertex in $S_L\cup S_R$ has its terminal vertex in $S_L\cup S_R$; if such an edge $e$ (labelled by an element of $A$) has initial vertex in $S_L^\$$ or $S_R^\$$, then it has its terminal vertex in the same set.  A directed edge $e$ labelled by $\$$ with initial vertex in $S_L$ has terminal vertex in $S_R^\$$, and similarly with $S_R$ and $S_L^\$$; if such an edge $e$ (labelled by $\$$) has initial vertex in $S_L^\$\cup S_R^\$$, then its terminal vertex is $s^\$$.

We say that $\calm$ \emph{accepts} a pair of strings $(W_L,W_R)\in A^\star \times A^\star$ if there is a shuffle $W$ of $(W_L\$,W_R\$)$ which is accepted by the automaton $\calm$.
\end{defn}

\subsection{Automaticity}
We are interested in studying when a group $\Gamma$ is asynchronously automatic; we briefly introduce the necessary definitions and basic results on the property below, and refer the interested reader to \cite{EpsteinEtAl1992} for a more comprehensive account.

Let $\Gamma$ be a group with a finite generating set $A$. We view $A$ as a finite set together with a function $\pi_A^0\colon A \to \Gamma$ that extends to a surjective monoid homomorphism $\pi_A\colon A^\star \to \Gamma$, where $A^\star$ is the free monoid on $A$.  We say that a word $W \in A^\star$ \emph{labels} or \emph{represents} the element $\pi_A(W) \in \Gamma$. For simplicity, we will assume that $A$ is symmetric, namely, $\pi_A(A) = \pi_A(A)^{-1}$, and contains the identity, that is $\pi_A(1) = 1_{\Gamma}$ for an element $1 \in A$. We denote by $d_A$ the combinatorial metric on the Cayley graph $\cay(\Gamma,A)$ of $\Gamma$.

We study combinatorial paths in $\cay(\Gamma,A)$. Given a path $P$ in $\cay(\Gamma,A)$ and an integer $t \in \{ 0, \ldots, |P| \}$, where $|P|$ is the length of $P$, we denote by $P(t) \in \Gamma$ the $t$-th vertex of $P$, so that $P(0)$ and $P(|P|)$ are the starting and ending vertices of $P$, respectively. We further define $P(t) \in \Gamma$ for any $t \in \ZZ_{\geq 0} \cup \{\infty\}$ by setting $P(t) = P(|P|)$ whenever $t > |P|$.

Given a word $W \in A^\star$, we denote by $\widehat{W}$ the path in $\cay(\Gamma,A)$ starting at $1_{\Gamma}$ and labelled by~$W$. In particular, for any $t \in \ZZ_{\geq 0}$, we write $\widehat{W}(t)$ for the element of $\Gamma$ represented by the prefix of $W$ of length $\min\{t,|W|\}$.

\begin{defn}[Asynchronously automatic group]
Let $\Gamma$ be a group with finite symmetric generating set $A$ containing the identity.  An \emph{asynchronous automatic structure} on $\Gamma$ consists of the set $A$, a finite state automaton $\calm$ over $A$, and asynchronous automata $\calm_x$ for $x\in A$ satisfying:
\begin{enumerate}
    \item The map $\pi_A|_{\call(\calm)}\colon \call(\calm)\to \Gamma$ is surjective;
    \item for $x\in A$, we have $(W_L,W_R)\in\call(\calm_x)$ if and only if $\pi_A(W_L x)=\pi_A(W_R)$ and both $W_L$ and $W_R$ are elements of $\call(\calm)$.
\end{enumerate}
We say $\Gamma$ is an \emph{asynchronously automatic group} if $\Gamma$ admits an asynchronous automatic structure.
\end{defn}

\begin{defn}[Departure function]
Let $\Gamma$ be a group with finite symmetric generating set $A$ containing the identity and let $\call$ be a regular language over $A$ that maps onto $\Gamma$.  A \emph{departure function}  for $(\Gamma,A,\call)$ is any function $\cald\colon \ZZ_{\geq 0} \to \ZZ_{\geq 0}$ such that if $W\in\call$, $r,s\geq 0$, $t\geq \cald(r)$, and $s+t\leq |W|$, then $d_A(\widehat{W}(s),\widehat{W}(s+t))\geq r$.
\end{defn}

The next definition is really a characterisation of boundedly asynchronous groups proven in \cite[Theorem~7.2.8]{EpsteinEtAl1992} (see also \cite[Proposition on p.~309]{Mosher1995}).

\begin{defn}[Boundedly asynchronous group]
Let $\Gamma$ be a group with finite symmetric generating set $A$ containing the identity and let $\call$ be a regular language over $A$ that maps onto $\Gamma$ under $\pi_A$.  We say $(A,\call)$ is a \emph{boundedly asynchronous structure} if
\begin{enumerate}[label=(\alph*)]
    \item there exists a departure function $\cald$ for $(\Gamma,A,\call)$;
    \item there exists a constant $\kappa > 0$, such that, for every pair of strings $V,W\in \call$ with $d_A(\pi_A(V),\pi_A(W))\leq 1$, we have that the Hausdorff distance of the paths $\widehat{V}$ and $\widehat{W}$ is at most $\kappa$.
\end{enumerate}
We say $\Gamma$ is \emph{boundedly asynchronous} if $\Gamma$ admits a boundedly asynchronous structure.
\end{defn}

The definitions of asynchronously automatic group and boundedly asynchronous group are equivalent:  every boundedly asynchronous group is immediately an asynchronously automatic group and by \cite[Theorem~7.2.4]{EpsteinEtAl1992} if a group admits an asynchronously automatic structure, then it also admits a boundedly asynchronous structure.

\subsection{Graphs of groups} \label{ssec:graphs-of-groups}

We are interested in groups acting on simplicial trees. We outline the main results we are using below, and refer the interested reader to \cite{SerreTrees} for a more comprehensive account.

We first introduce the notion of graphs of groups and their fundamental groups. Given a finite undirected graph $\calg$ (with loops and multiple edges allowed), we write $V(\calg)$ and $E^+(\calg)$ for its sets of vertices and directed edges, respectively. We also write $E^+(\calg) \to E^+(\calg), e \mapsto \overline{e}$ for the function changing the direction of edges, and $\iota\colon E^+(\calg) \to V(\calg)$ for the starting vertex function, so that an edge $e$ has endpoints $\iota(e)$ and $\iota(\overline{e})$.

\begin{defn}[Fundamental group of graph of groups]
A \emph{graph of groups} is a connected finite undirected graph $\calg$ together with a collection of groups $\{ \Gamma_v \mid v \in V(\calg) \}$, a collection of groups $\{ \Gamma_e \mid e \in E^+(\calg) \}$ satisfying $\Gamma_e = \Gamma_{\overline{e}}$, and a collection of injective homomorphisms $\{ i_e\colon \Gamma_e \to \Gamma_{\iota(e)} \mid e \in E^+(\calg) \}$.

Given a graph of groups $(\calg,\{\Gamma_v\},\{\Gamma_e\},\{i_e\})$ as above, its \emph{fundamental group} is the group generated by $\left( \bigsqcup_{v \in V(\calg)} \Gamma_v \right) \sqcup E^+(\calg)$ with defining relations $\overline{e} = e^{-1}$ for $e \in E^+(\calg)$, $i_{\overline{e}}(g) = e^{-1} i_e(g) e$ for $e \in E^+(\calg)$ and $g \in \Gamma_e$, and $e = 1$ for $e \in E^+(T_{\calg})$, where $T_{\calg}$ is a maximal subtree of $\calg$.
\end{defn}

Let $(\calg,\{\Gamma_v\},\{\Gamma_e\},\{i_e\})$ be a graph of groups, and let $\Gamma$ be its fundamental group. Given $e \in E^+(\calg)$, let $\cals(e)$ be a left transversal of $\Gamma_e$ in $\Gamma_{\iota(e)}$ containing $1_\Gamma \in \Gamma$; pick also a base vertex $c \in V(\calg)$. Let $\calw_{\calg}$ be the set of all expressions of the form $s_1 e_1 s_2 e_2 \cdots s_n e_n$, where $e_1 e_2 \cdots e_n$ is a directed loop in $\calg$ based at~$c$, and $s_i \in \cals(e_i)$ with $s_i \neq 1_\Gamma$ whenever $e_{i-1} = \overline{e_i}$. It is then well-known \cite[Theorem~11 on p.~45]{SerreTrees} that every element $g \in \Gamma$ has a unique expression of the form $g = Uh$ with $U\in \calw_{\calg}$ and $h \in \Gamma_c$. We call such an expression the \emph{normal form} for $g$.

Now let $\Gamma$ be a group acting cocompactly on a simplicial tree $\calt$. Without loss of generality (subdividing edges of $\calt$ if necessary), we may assume that the action is without edge inversions, i.e.\ any element of $\Gamma$ fixing an edge of $\calt$ also fixes its endpoints. In that case, using Bass--Serre theory (see \cite[Theorem~13 on p.~55]{SerreTrees}), $\Gamma$ can be described as the fundamental group of a graph of groups $(\calg,\{\Gamma_v\},\{\Gamma_e\},\{i_e\})$, where $\calg = \calt/\Gamma$, the groups $\Gamma_v$ and $\Gamma_e$ are the stabilisers (under the action of $\Gamma$) of lifts $\widetilde{v} \in V(\calt)$ and $\widetilde{e} \in E^+(\calt)$ of $v \in V(\calg)$ and $e \in E^+(\calg)$, respectively, and $i_e\colon \Gamma_e \to \Gamma_{\iota(e)}$ is the inclusion of $\Gamma_e$ into the stabiliser of $\iota(\widetilde{e})$ composed with an inner automorphism of $\Gamma$. Moreover, if $g = Uh$ is a normal form of $g \in \Gamma$, where $U = s_1 e_1\cdots s_n e_n \in \calw_{\calt}$ and $h \in \Gamma_c$, then $e_1 \cdots e_n$ is the projection of the geodesic path in $\calt$ from $\widetilde{c}$ to $g \cdot \widetilde{c}$; here and later, we write $\calw_\calt$ for $\calw_\calg$.

\section{Proof of Proposition~\ref{prop:main}} \label{sec:proof}

Suppose a group $\Gamma$ acts cocompactly on a locally finite tree $\calt$ without edge inversions (so that we can use the notation of Section~\ref{ssec:graphs-of-groups}), and suppose that the vertex stabilisers under this action are (Gromov) hyperbolic. Let $B$ be a finite symmetric generating set of $\Gamma_c$, and write $\calw_B \subseteq B^\star$ for the set of all geodesic words over $B$. Consider the generating set $A' := E^+(\calg) \sqcup B \sqcup \bigsqcup_{e \in E^+(\calg)} \cals(e)$ for $\Gamma$ and its ``symmetric closure'' $A = A' \sqcup (A')^{-1}$; note that since $\calt$ is locally finite, we have $|\cals(e)| < \infty$ for all $e \in E^+(\calg)$, and hence $A'$ (and therefore $A$) is finite. Consider the language
\[
\call \coloneqq \{ U_{\calt} U_B \mid U_{\calt} \in \calw_{\calt}, U_B \in \calw_B \}
\]
of words over $A$.

We note that $\call$ is a regular language over $A$: see Figure~\ref{fig:automaton}. We also note that there are finitely many words in $\call$ representing any given element of $g \in \Gamma$, since the normal form for $g$ is unique and there are finitely many geodesic words over $B$ representing any given element of $\Gamma_c$. We claim that $(A,\call)$ is an boundedly asynchronous structure for $\Gamma$.

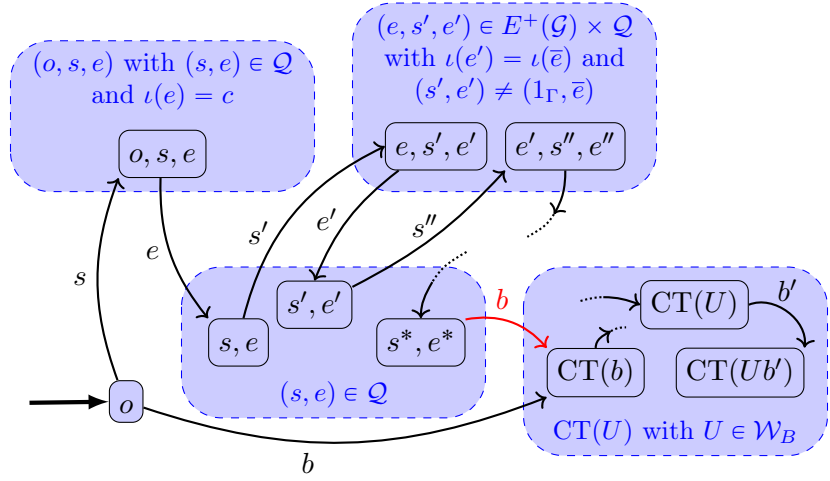
\begin{figure}[ht]
    \centering
\begin{tikzpicture}
\draw [blue,dashed,fill=blue!20,rounded corners=15] (-1,1) rectangle (3,3);
\node [blue,below] at (1,3) {\parbox{4cm}{\centering\small $(o,s,e)$ with $(s,e) \in \calq$ and $\iota(e)=c$}};
\draw [blue,dashed,fill=blue!20,rounded corners=15] (1.25,0) rectangle (5.25,-2);
\node [blue,above] at (3.25,-2) {\parbox{4cm}{\centering\small $(s,e) \in \calq$}};
\draw [blue,dashed,fill=blue!20,rounded corners=15] (3.5,1) rectangle (7.5,3.5);
\node [blue,below] at (5.5,3.5) {\parbox{4cm}{\centering\small $(e,s',e') \in E^+(\calg) \times \calq$ with $\iota(e')=\iota(\overline{e})$ and $(s',e') \neq (1_{\Gamma},\overline{e})$}};
\draw [blue,dashed,fill=blue!20,rounded corners=15] (5.75,0) rectangle (9.75,-2.5);
\node [blue,above] at (7.75,-2.5) {\parbox{4cm}{\centering\small $\CT(U)$ with $U \in \calw_B$}};
\node [below left,draw,rounded corners,fill=blue!20] (o) at (0.75,-1.5) {$\vphantom{'}o$};
\node [draw,rounded corners] (ose) at (1,1.5) {$\vphantom{'}o,s,e$};
\node [draw,rounded corners] (se1) at (2,-1) {$\vphantom{'}s,e$};
\node [draw,rounded corners] (se2) at (3,-0.5) {$s',e'$};
\node [draw,rounded corners] (se3) at (4.4,-1) {$\vphantom{'}s^*,e^*$};
\node [draw,rounded corners] (ese1) at (4.6,1.6) {$e,s',e'$};
\node [draw,rounded corners] (ese2) at (6.3,1.6) {$e',s'',e''$};
\node [draw,rounded corners] (ct1) at (6.7,-1.4) {$\CT(b)$};
\node [draw,rounded corners] (ct2) at (8,-0.5) {$\CT(U)$};
\node [draw,rounded corners] (ct3) at (8.6,-1.4) {$\CT(Ub')$};
\draw [ultra thick,-{latex}] (-0.75,-1.8) -- (o);
\draw [thick,->] (o) to[bend left=20] node [midway,left] {$s$} (ose.south west);
\draw [thick,->] (ose) to[bend right=20] node [midway,left] {$e$} (se1.north west);
\draw [thick,->] (se1) to[bend left=25] node [pos=0.4,left] {$s'$} (ese1.west);
\draw [thick,->] (ese1) to[bend right=15] node [midway,left] {$e'$} (se2.north);
\draw [thick,->] (se2) to[bend right=10] node [pos=0.6,left] {$s''$} (ese2.south west);
\draw [thick,->] (ese2) to[out=-90,in=60] (6.15,0.7);
\draw [thick,densely dotted] (6.15,0.7) to[out=-120,in=30] (5.8,0.4);
\draw [thick,densely dotted] (5,0.2) to[out=-150,in=60] (4.55,-0.2);
\draw [thick,->] (4.55,-0.2) to[out=-120,in=90] (se3.north);
\draw [thick,->,red] (se3) to[bend left] node [pos=0.4,above] {$b$} (ct1.north west);
\draw [thick,->] (o) to[bend right=20] node [pos=0.4,below] {$b$} (ct1.south west);
\draw [thick,->] (ct1) to[out=90,in=-150] (6.9,-0.85);
\draw [thick,densely dotted] (6.9,-0.85) to[out=30,in=180] (7.1,-0.8);
\draw [thick,densely dotted] (6.5,-0.45) to[out=15,in=180] (6.8,-0.4);
\draw [thick,->] (6.8,-0.4) to[out=0,in=160] (ct2.west);
\draw [thick,->] (ct2.east) to[bend left=60] node [midway,above] {$b'$} (ct3.north east);
\end{tikzpicture}
    \caption{A finite state automaton over $A$ accepting $\call$. Here we set $\calq := \{ (s,e) \mid e \in E^+(\calg),s \in \cals(e) \}$. We use $b,b'$ for letters in $B$, and the red arrow exists if and only if $\iota(\overline{e^*}) = c$. We set $\CT(U) := \{ V \in B^* \mid UV \in \calw_B \}$ to be the \emph{cone type} of~$U$; it is well-known that since $\Gamma_c$ is hyperbolic it has finitely many cone types \cite[Theorem~3.2.1]{EpsteinEtAl1992}. The initial state of the automaton is $(o)$, and the final states are $(o)$, $(s,e) \in \calq$ with $\iota(\overline{e}) = c$, and $\CT(U)$ for $U \in \calw_B$.}
    \label{fig:automaton}
\end{figure}

\begin{lemma}
\label{lem:departure}
There exists a departure function $\cald$ for $(\Gamma,A,\call)$.
\end{lemma}

\begin{proof}
Let $\calm$ be a finite state automaton over $A$ accepting $\call$. Suppose, without loss of generality (removing states from $\calm$ if necessary), that each state $s \in \calm$ is accessible (there exists a word $U_s \in A^\star$ labelling a path from the initial state of $\calm$ to $s$) and live (there exists a word $W_s \in A^\star$ labelling a path from $s$ to a final state of $\calm$). Given any two states $s,t \in V(\calg(\calm))$ and an element $g \in \Gamma$, write $\calv_{s,t}(g)$ for the set of words $V \in A^\star$ labelling a path in $\calm$ from $s$ to $t$ such that $\pi_A(V) = g$.

Now, as there are finitely many words in $\call$ representing any given element of $\Gamma$, it follows that $\calv_{s,t}(g)$ is finite for all $s$, $t$ and $g$, since $\{ U_s V W_t \mid {V \in \calv_{s,t}(\gamma)} \}$ is a collection of words in $\call$ representing $\pi_A(U_s) g \pi_A(W_t) \in \Gamma$. In particular, since $A$ and the number of states in $\calm$ are finite, the set $\calu_r := \bigcup \{ \calv_{s,t}(g) \mid s,t \in V(\calg(\calm)), g \in \Gamma, d_A(1,g) < r \}$ is also finite for any $r \geq 0$. We can therefore define a function $\cald\colon \ZZ_{\geq 0} \to \ZZ_{\geq 0}$ by setting $\cald(r)$ to be larger than the length of any word in~$\calu_r$. It is then clear from the construction that $\cald$ is a departure function for $(\Gamma,A,\call)$, as required.
\end{proof}

\begin{prop} \label{prop:Hausdorff}
There exists a constant $\kappa>0$ with the following property: for every $V,W\in \call$ with $d_A(\pi_A(V),\pi_A(W))\leq 1$, the Hausdorff distance between $\widehat{V}$ and $\widehat{W}$ is at most $\kappa$.
\end{prop}

The idea of the proof is as follows: first, it can be shown that any vertex of $\widehat{W}$ that is not a vertex of $\widehat{V}$ is bounded distance away from the left coset $\pi_A(V)\Gamma_c$. The resulting vertices of $\pi_A(V)\Gamma_c$ therefore form a quasi-geodesic, and the result then follows from stability of quasi-geodesics in the hyperbolic group $\Gamma_c$. The proof of Proposition~\ref{prop:Hausdorff} is illustrated in Figure~\ref{fig:Hausdorff}.

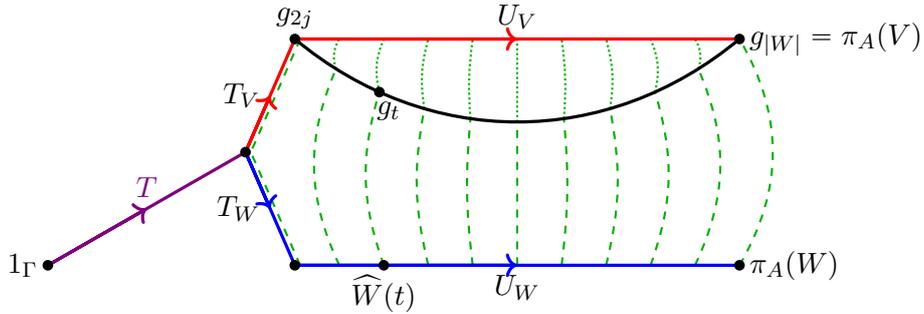
\begin{figure}[ht]
{\centering
\begin{tikzpicture}[x=1.3cm,y=1.5cm]
\draw [thick,green!70!black,dashed,xshift=2pt] (2.5,0) -- (2,1) -- (2.5,2);
\draw [red!50!blue,very thick,->] (2,1) -- (0,0) -- (1,0.5) node [above] {$T$};
\path (2.5,2) -- (2,1) -- (2.25,1.5) node [left] {$T_V$};
\path (7,2) -- (2.5,2)
    node [behind path, pos=0.9] (c1) {}
    node [behind path, pos=0.8] (c2) {}
    node [behind path, pos=0.7] (c3) {}
    node [behind path, pos=0.6] (c4) {}
    node [behind path, pos=0.5] (c5) {}
    node [behind path, pos=0.4] (c6) {}
    node [behind path, pos=0.3] (c7) {}
    node [behind path, pos=0.2] (c8) {}
    node [behind path, pos=0.1] (c9) {}
    node [behind path, pos=0] (c10) {}
    -- (4.75,2) node [above] {$U_V$};
\path (2.5,0) -- (2,1) -- (2.25,0.5) node [left] {$T_W$};
\path  (7,0) -- (2.5,0)
    node [behind path, pos=0.9] (a1) {}
    node [behind path, pos=0.8] (a2) {}
    node [behind path, pos=0.7] (a3) {}
    node [behind path, pos=0.6] (a4) {}
    node [behind path, pos=0.5] (a5) {}
    node [behind path, pos=0.4] (a6) {}
    node [behind path, pos=0.3] (a7) {}
    node [behind path, pos=0.2] (a8) {}
    node [behind path, pos=0.1] (a9) {}
    node [behind path, pos=0] (a10) {}
    -- (4.75,0) node [below] {$U_W$};
\path (7,2) to[bend left=40] 
    node [behind path, pos=0.9] (b1) {}
    node [behind path, pos=0.8] (b2) {}
    node [behind path, pos=0.7] (b3) {}
    node [behind path, pos=0.6] (b4) {}
    node [behind path, pos=0.5] (b5) {}
    node [behind path, pos=0.4] (b6) {}
    node [behind path, pos=0.3] (b7) {}
    node [behind path, pos=0.2] (b8) {}
    node [behind path, pos=0.1] (b9) {}
    node [behind path, pos=0] (b10) {}
    (2.5,2);
\foreach \i in {1,...,10} {
    \draw [thick,green!70!black,dashed] (a\i.center) to[bend right={6*\i-30}] (b\i.center);
    \draw [thick,green!70!black,densely dotted] (b\i.center) to[bend right={6*\i-30}] (c\i.center);
}
\draw [blue,very thick,->] (2.5,0) -- (2,1) -- (2.25,0.5);
\draw [blue,very thick,->] (7,0) -- (2.5,0) -- (4.75,0);
\draw [red,very thick,->] (2.5,2) -- (2,1) -- (2.25,1.5);
\draw [red,very thick,->] (7,2) -- (2.5,2) -- (4.75,2);
\fill (0,0) circle (2pt) node [left] {$1_\Gamma$};
\fill (7,0) circle (2pt) node [right] {$\pi_A(W)$};
\fill (7,2) circle (2pt) node [right] {$g_{|W|} = \pi_A(V)$};
\fill (2.5,2) circle (2pt) node [above] {$g_{2j}$};
\fill (b2) circle (2pt) node [below,xshift=3.5pt] {$g_t$};
\fill (a2) circle (2pt) node [below] {$\widehat{W}(t)$};
\draw [very thick] (7,2) to[bend left=40] (2.5,2);
\fill (2,1) circle (2pt);
\fill (2.5,0) circle (2pt);
\end{tikzpicture}

}

\caption{The proof of Proposition~\ref{prop:Hausdorff}, with $\widehat{V}$ (red and purple) and $\widehat{W}$ (blue and purple) shown. The black path is a $(\zeta,\zeta)$-quasi-geodesic in $\pi_A(V)\Gamma_c$ (with respect to the metric $d_B$), the green dashed lines represent paths of length $\leq 2\eta$ in $\cay(\Gamma,A)$, and the green dotted lines are mapped by $\pi_A(V)^{-1}$ to paths of length $\leq \kappa-2\eta$ in $\cay(\Gamma_c,B)$.}
\label{fig:Hausdorff}
\end{figure}

\begin{proof}
Let $\eta = \max \{ d_{\calt}(\widetilde{c},a \cdot \widetilde{c}) \mid a \in A \}$, and let $\zeta \geq 4\eta+1$ be such that for any $g \in \Gamma_c$ with $d_A(1,g) \leq 4\eta+1$ we have $d_B(1,g) \leq \zeta$ (such $\eta$ and $\zeta$ exist since $A$ is finite). Since $\Gamma_c$ is hyperbolic, there exists $\kappa \geq 2\eta$ such that any two $(\zeta,\zeta)$-quasi-geodesics with the same endpoints in the Cayley graph $\cay(\Gamma_c,B)$ are Hausdorff distance $\leq \kappa-2\eta$ away from each other \cite[Chapitre 3, Th\'eor\`eme 1.3]{CDP2006}.

Now, we can write $V = s_1e_1 \cdots s_ne_n U_V$ and $W = s_1'e_1' \cdots s_m'e_m' U_W$, where $s_1e_1 \cdots s_ne_n, s_1'e_1' \cdots s_m'e_m' \in \calw_{\calt}$ and $U_V,U_W \in \calw_B$. Let $j \leq \min\{n,m\}$ be the largest integer such that $(s_i,e_i) = (s_i',e_i')$ for all $i \leq j$; we write $T = s_1e_1 \cdots s_je_j$, $T_V = s_{j+1}e_{j+1} \cdots s_ne_n$ and $T_W = s_{j+1}'e_{j+1}' \cdots s_m'e_m'$, so that $V = T T_V U_V$, and $W = T T_W U_W$. Since $d_A(\pi_A(V),\pi_A(W))\leq 1$, it then follows that $d_{\calt}(\widetilde{c}, \pi_A(T_V^{-1} T_W) \cdot \widetilde{c}) = d_{\calt}(\pi_A(V) \cdot \widetilde{c}, \pi_A(W) \cdot \widetilde{c}) \leq \eta$, and therefore $(n-j)+(m-j) \leq \eta$.

We aim to show that for any $t \in \ZZ_{\geq 0}$, the element $\widehat{W}(t)$ is distance $\leq \kappa$ away from $\widehat{V}$ in $\cay(\Gamma,A)$; the proof is analogous if the roles of $\widehat{V}$ and $\widehat{W}$ are swapped. The result is clear for $t < 2j$, as in that case $\widehat{W}(t) = \widehat{V}(t)$. Therefore, we may assume, without loss of generality, that $t \geq 2j$.

We first claim that for each $t$ there exists an element $g_t \in \pi_A(V) \Gamma_c$ such that we have $d_A(\widehat{W}(t),g_t) \leq 2\eta$. Indeed, if $t \leq 2m$ then we can take $g_t := \pi_A(T T_V)$: then $g_t^{-1}\widehat{W}(t)$ is labelled by a subword of $T_V^{-1}T_W$, and we have $|T_V| = 2(n-j)$ and $|T_W| = 2(m-j)$, implying that $d_A(\widehat{W}(t),g_t) \leq 2(n-j)+2(m-j) \leq 2\eta$. Otherwise, we have $\widehat{W}(t) \in \pi_A(W)\Gamma_c$ and therefore $d_A(\widehat{W}(t)^{-1}\pi_A(V) \cdot \widetilde{c}, \widetilde{c}) = d_A(\pi_A(V) \cdot \widetilde{c},\pi_A(W) \cdot \widetilde{c}) \leq \eta$, implying that $\widehat{W}(t)^{-1}\pi_A(V)$ has normal form $T' U'$, where $T' = s_1''e_1'' \cdots s_\ell''e_\ell'' \in \calw_{\calt}$ for some $\ell \leq \eta$ and $U' \in \calw_B$; we then set $g_t := \widehat{W}(t)^{-1} \pi_A(T')$, so that $d_A(\widehat{W}(t),g_t) = d_A(1,\pi_A(T')) \leq 2\ell \leq 2\eta$, as claimed. Without loss of generality, we may also assume that $g_{|W|} = \pi_A(V)$, since we have $d_A(\pi_A(V),\widehat{W}(|W|)) \leq 1$.

Now let $g_t' := \pi_A(TT_V)^{-1} g_t $ for $2j \leq t \leq |W|$. This yields a collection $g_{2j}',g_{2j+1}',\ldots,g_{|W|}' \in \Gamma_c$ such that $d_A(\widehat{W}(t),\pi_A(TT_V)g_t') \leq 2\eta$ for all $t$. In particular, for $2j \leq t < |W|$ we have
\begin{align*}
d_A(g_t',g_{t+1}') &= d_A(g_t,g_{t+1}) \\ &\leq d_A(g_t,\widehat{W}(t))+d_A(\widehat{W}(t),\widehat{W}(t+1))+d_A(\widehat{W}(t+1),g_{t+1}) \\ &\leq 2\eta + 1 + 2\eta = 4\eta+1,
\end{align*}
implying that $d_B(g_t',g_{t+1}') \leq \zeta$ by the choice of $\zeta$. In particular, the points $g_{2j}',\ldots,g_{|W|}'$ are vertices of a $(\zeta,\zeta)$-quasi-geodesic in $\cay(\Gamma_c,B)$ starting at $g_{2j}' = 1_{\Gamma_c}$ and ending at $g_{|W|}' = \pi_A(U_V)$. Therefore, by the choice of $\kappa$, the element $g_t'$ is distance at most $ \kappa-2\eta$ away from $\widehat{U_V}$ in $\cay(\Gamma_c,B)$, and hence in $\cay(\Gamma,A)$. This implies that $g_t = \pi_A(TT_V)g_t'$ is distance at most $ \kappa-2\eta$ away from $\pi_A(TT_V)\widehat{U_V} \subseteq \widehat{V}$, and thus $\widehat{W}(t)$ is distance at nost $ (\kappa-2\eta)+d_A(\widehat{W}(t),g_t) \leq \kappa$ away from $\widehat{V}$, as required.
\end{proof}

We have verified all of the properties for $(A,\call)$ to be a boundedly asynchronous structure for $\Gamma$. \qed

\section{Proof of Corollary~\ref{cor lattices}}\label{sec cor proof}

\begin{proof}
Let $\calt$ denote a locally finite tree such that $T=\Aut(\calt)$.  By \cite[Theorem~A(1)]{Hughes2021} the lattice $\Gamma$ splits as a \emph{graph of $H$-lattices}.   We briefly explain what this means.  By \cite[Definition~3.2]{Hughes2021}, $\Gamma$ splits as finite graph of groups such that every vertex or edge group $\Gamma_\sigma$ fits into a short exact sequence
\[\begin{tikzcd}
    1 \arrow[r] & F_\sigma  \arrow[r] & \Gamma_\sigma \arrow[r] & \Lambda_\sigma \arrow[r] & 1,
\end{tikzcd}\]
where $F_\sigma$ is a finite group and $\Lambda_\sigma$ is a uniform lattice in $H$.  Note that the graph of groups decomposition is induced by the action on the locally finite tree $\calt$.  Thus, to apply \Cref{prop:main} it suffices to show the vertex and edge groups are hyperbolic.  Now, any vertex or edge group $\Gamma_\sigma$ acts properly and cocompactly on the associated symmetric space $X$ of $H$ (via the projection $\Gamma_\sigma\onto \Lambda_\sigma$).  Here $X$ is one of the negatively curved symmetric spaces $\RH^n$, $\CH^n$, $\HH^n$, or $\OH$.  In particular, $X$ is a hyperbolic space in the sense of Gromov.  Hence, the vertex and edge groups are hyperbolic groups.   The result follows from \Cref{prop:main}.
\end{proof}

\bibliographystyle{halpha}
\bibliography{refs.bib}

\end{document}